\documentclass[11pt]{amsart}
\usepackage[bookmarks=false]{hyperref}
\usepackage{amsfonts,amssymb,amsmath,amsthm}
\usepackage{bbm}
\usepackage{graphicx, import}
\usepackage{transparent}
\usepackage{caption}
\usepackage[labelfont=normalfont,labelformat=simple]{subcaption}

\usepackage{enumitem}
\usepackage{aliascnt} 
\usepackage{fancyhdr,url}
\usepackage[usenames,dvipsnames]{xcolor}
\usepackage{tikz}
\usepackage[all]{xy}
\hypersetup{colorlinks=true,linkcolor=Cyan,citecolor=Cyan}

\makeatletter
\def\MyNewTheorem#1[#2]#3{%
  \newaliascnt{#1}{#2}
  \newtheorem{#1}[#1]{#3}
  \aliascntresetthe{#1}
  \expandafter\newcommand\csname #1autorefname\endcsname{#3}
}
\makeatother

\makeatletter
\newtheorem*{rep@theorem}{\rep@title}
\newcommand{\newreptheorem}[2]{%
\newenvironment{rep#1}[1]{%
 \def\rep@title{#2 \ref{##1}}%
 \begin{rep@theorem}}%
 {\end{rep@theorem}}}
\makeatother

\newtheorem{theorem}{Theorem}
\newtheorem{question}{Question} 
\MyNewTheorem{cor}[theorem]{Corollary}
\MyNewTheorem{lemma}[theorem]{Lemma}
\newreptheorem{theorem}{Theorem}

\theoremstyle{definition}
\MyNewTheorem{definition}[theorem]{Definition}

\numberwithin{equation}{section}
\def\equationautorefname~#1\null{(#1)\null}
\def\itemautorefname~#1\null{#1\null}

\title{Trisections of $4$--manifolds with Boundary}
\author{Nickolas A. Castro \and David T. Gay \and Juanita Pinz\'on-Caicedo}

\thanks{The first author is supported by the UC Davis Chancellor's Postdoctoral Fellowship. The second author was supported by grant \#359873 from the Simons Foundation. The second and third author are supported by NSF grant DMS-1664567.}

\begin{document}
\begin{abstract}
Given a handle decomposition of a 4-manifold with boundary, and an open book decomposition of the boundary, we show how to produce a trisection diagram of a trisection of the 4-manifold inducing the given open book. We do this by making the original proof of the existence of relative trisections more explicit, in terms of handles. Furthermore, we extend this existence result to the case of 4-manifolds with multiple boundary components, and show how trisected 4-manifolds with multiple boundary components glue together.
\end{abstract}

\maketitle

When developing the foundations of the study of trisections of $4$--manifolds in~\cite{gk}, Gay and Kirby briefly discussed the case of $4$--manifolds with connected boundary, in which the corresponding boundary data are open book decompositions on $3$--manifolds. Castro~\cite{castro} developed this case further, in particular demonstrating the importance of this case by showing how to glue two trisected $4$--manifolds along a common boundary to produce a trisection of a closed manifold. The current authors then worked out~\cite{cgp} the diagrammatic version of this theory, with the diagrammatic version of gluing appearing in Castro's work with Ozbagci~\cite{co}. This paper supplements these papers by first showing, through examples and a careful exposition of the existence proof in~\cite{gk}, how to explicitly turn a handle decomposition of a $4$--manifold together with the data of a given open book on the boundary into a trisection of the $4$--manifold described diagrammatically. Then we extend the definitions to include trisections of $4$--manifolds with multiple boundary components and discuss the corresponding theorems in that case. We end by implementing the method discussed earlier for turning handle decompositions into diagrams in this multiple boundary setting, showing how to get trisection diagrams for product cobordisms.


\begin{section}{Transforming Kirby Diagrams into Relative Trisection Diagrams}

An open book decomposition of a closed, oriented, connected 3--manifold $Y$ is an oriented link $\Lambda\subset Y$ and a fibration $f:Y\setminus \Lambda \to S^1$ such that, for each $t \in S^1$ we get a compact oriented surface $P_t = f^{-1}(t)$ with $\partial P_t = \Lambda$. If $Y = \partial X$ for a compact, connected $4$--manifold $X$, there is a notion defined in~\cite{gk} of a {\em relative trisection} of $X$ inducing the given open book on $Y$. In \autoref{sec:basicdefs} below we give a more general version of this definition in the case of multiple boundary components for $X$, but for the purposes of this section we take the definition as understood, and the reader unfamiliar with the definition will gain an understanding by example, as preparation for \autoref{sec:basicdefs}. (Such a reader may also choose to begin by reading \autoref{sec:basicdefs}.) In~\cite{cgp}, the present authors show how to diagrammatically characterize trisections on $4$--manifolds with connected boundary and how to understand the open book on the boundary. 
\begin{question} \label{question}
 Given a handle decomposition of a compact $4$--manifold $X$ with boundary $Y$, involving a single $0$--handle, some $1$--, $2$-- and $3$--handles, described diagrammatically via a Kirby diagram (which is thus also a surgery diagram for $Y$), and given an open book decomposition of $Y$ described by explicitly drawing a single page in the surgery diagram for $Y$, how does one produce a trisection diagram for $X$ inducing the given open book decomposition on $Y$?
\end{question}
We answer this by giving a detailed proof paired with an extended example of the basic existence theorem:
\begin{theorem}[Gay-Kirby~\cite{gk}]
 Given an open book decomposition of $Y = \partial X$ there exists a relative trisection of $X$ inducing the given open book. (Here $Y$ and $X$ are connected.)
\end{theorem}
We show that this existence theorem can in fact be made explicit so as to produce a diagram for the trisection, in the sense of  \autoref{question}.

\begin{proof}

We begin with an example of the given data in \autoref{question}. For $X$ we take the complement of the standard torus in $S^4$, and for the open book on $Y=T^3$ we take a standard open book with a twice punctured torus as the page, and monodromy given by a right handed Dehn twist parallel to one boundary component and a left handed twist parallel to the other boundary component. \autoref{fig::T2complementKirby} shows a Kirby diagram for $X$, once drawn with dotted circle notation and once with ball notation for the $1$--handles.

\begin{figure}[h]
\centering
\def\svgwidth{0.2\textwidth}

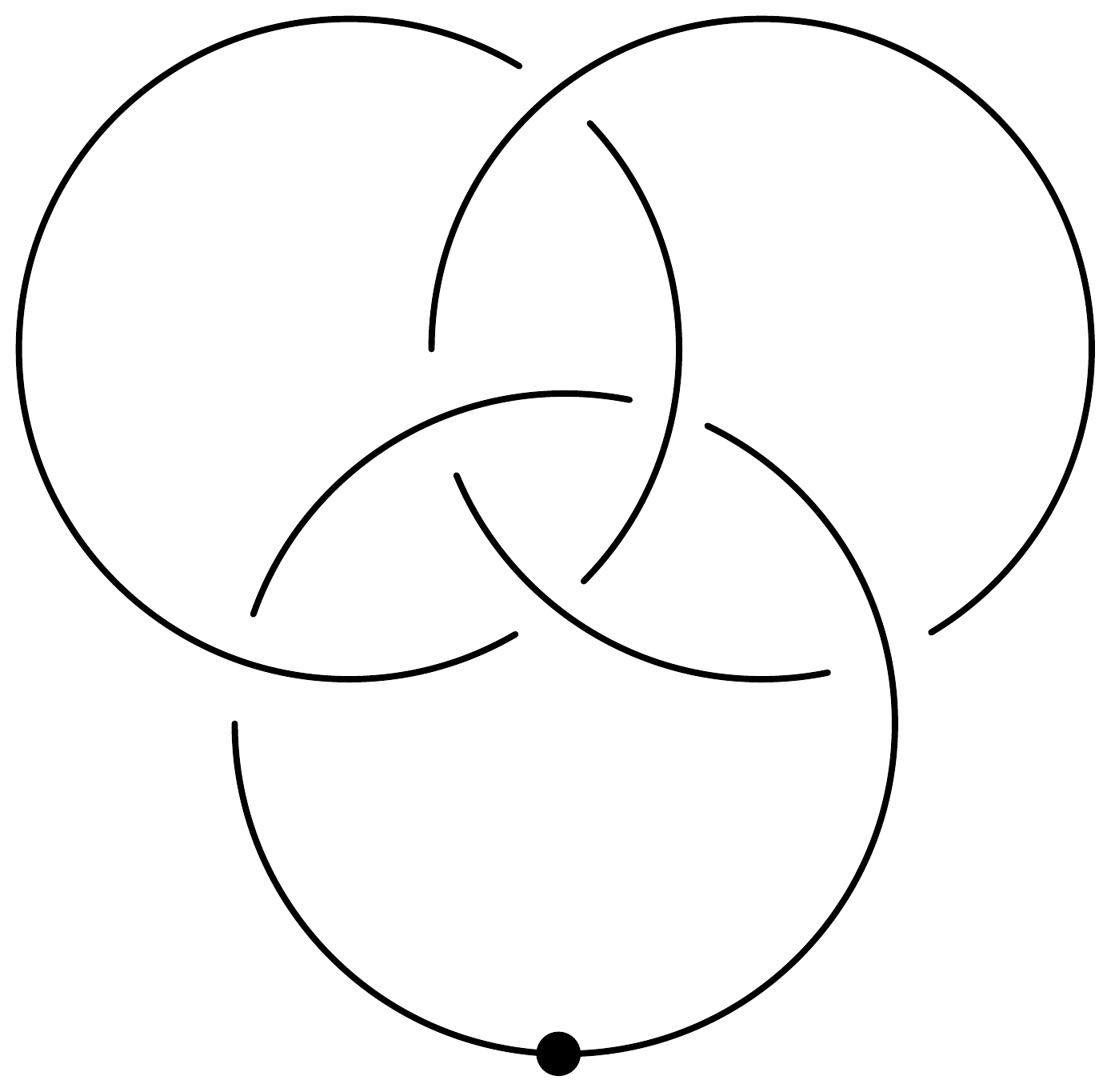
\def\svgwidth{0.35\textwidth}
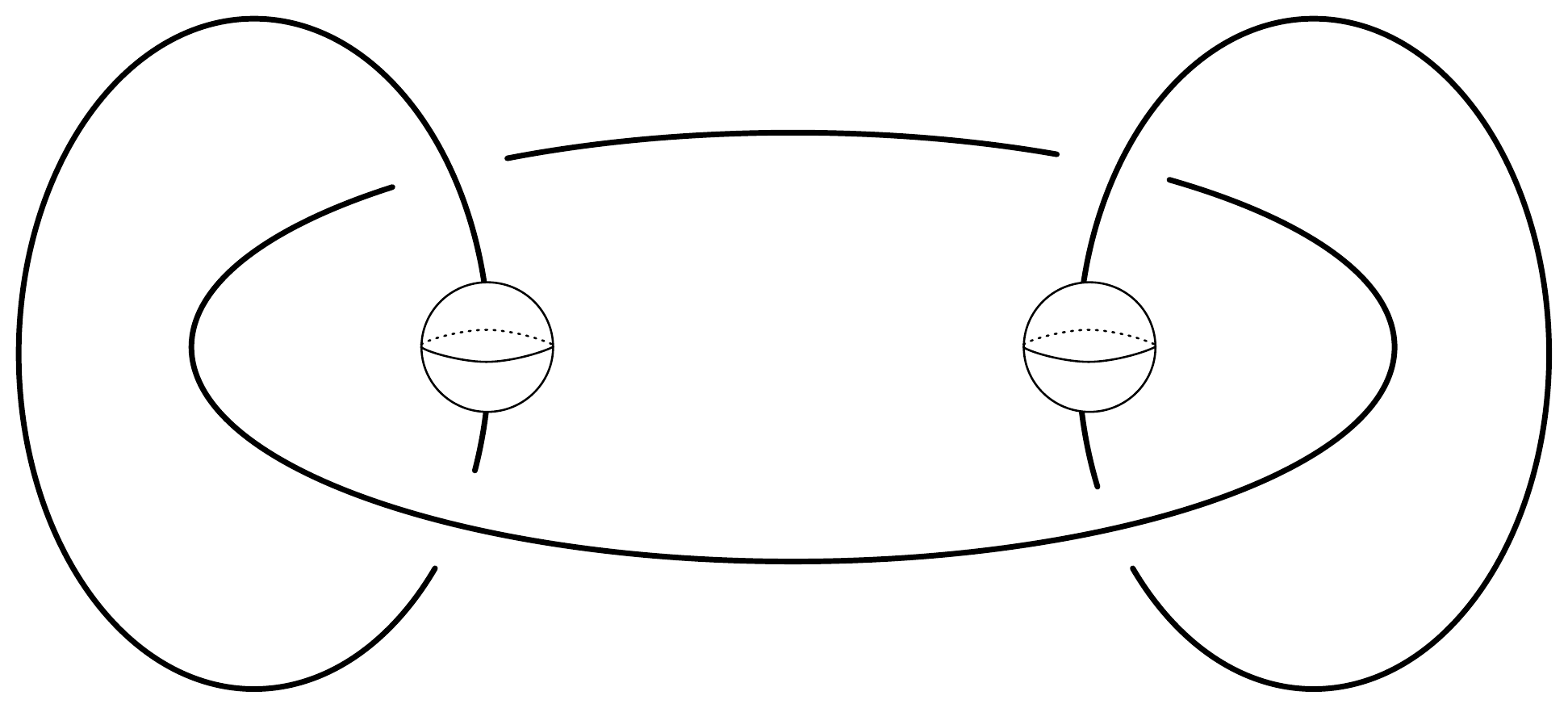
\caption{Kirby diagram for a handle decomposition of $X=S^4\setminus N(T^2)$ with one $0$--handle, one $1$--handle and two $2$--handles. The $2$--handles are $0$--framed. \label{fig::T2complementKirby}}
\end{figure}

We need to augment this diagram with an explicit embedding of a page of the open book decomposition we have in mind.
The top diagram in \autoref{fig::fibration} shows a generic fiber $F$ of the fibration of $T^3$ over $T^2$, and its location within the surgery description of $\partial X$. This fibration structure can be transformed into the open book under consideration via the three component fibered link $C_0\amalg C_1\amalg C_2$ obtained by plumbing a positive Hopf band to a negative Hopf band along a boundary parallel arc. To be precise, we remove a disk $D$ with boundary from the fiber $F$ of the fibration and then realize $T^3$ as 
$$(F\setminus D) \times S^1) \underset{\partial D=C_0}{\cup} (S^3\setminus \nu(C_0)).$$ 
This decomposition gives the right open book for $T^3$. The page $P$ for this open book embedded into the surgery diagram is shown in the middle diagram in \autoref{fig::fibration}, and an ambient isotopy yields the final diagram. Note that the final diagram depicts a handle decomposition of $P$ involving one $0$--handle and three $1$--handles; this takes us closer to a ``planar diagram'' of the page which will be convenient at the next stage.

\begin{figure}[h]
\centering
\def\svgwidth{0.3\textwidth}
\input{kirby_fibration.pdf_tex}
\def\svgwidth{0.3\textwidth}
\input{kirby_open_book.pdf_tex}
\def\svgwidth{0.3\textwidth}
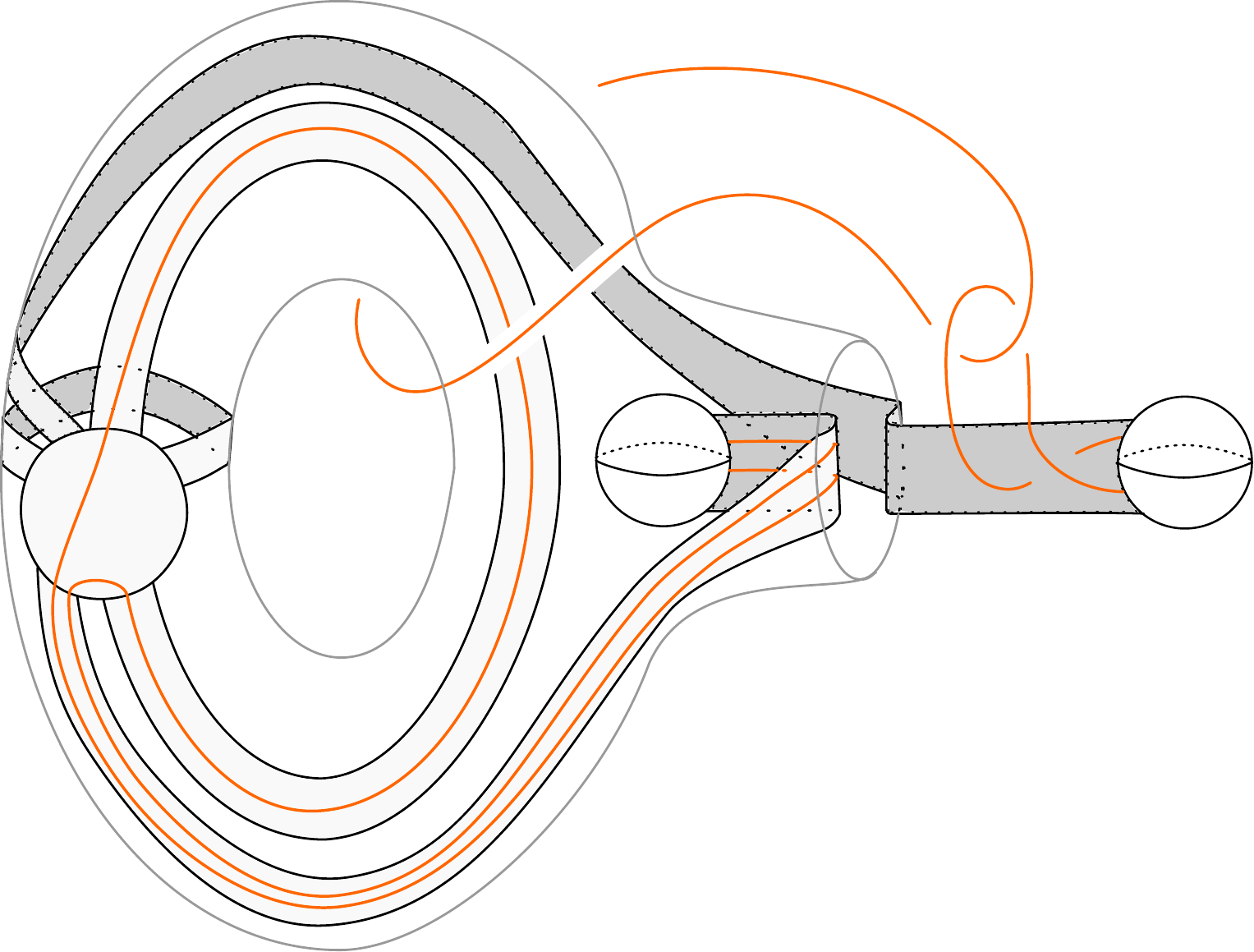
\caption{Fibration and open book for $\partial X\cong T^3$.}
\label{fig::fibration}
\end{figure}

With $X$ and $Y$ described by a Kirby diagram, and the open book described via an embedded page $P$, the first step is to see $X$ as a relative cobordism between $3$--manifolds with boundary and produce a new handle decomposition of $X$ adapted to this cobordism structure.

\begin{definition} Let $Y_-$ and $Y_+$ be 3-dimensional oriented manifolds with boundary, and $X$ a 4--dimensional oriented manifold with boundary. The triple $(X ; Y_-, Y_+)$ is a relative cobordism between $Y_-$ and $Y_+$ if
\begin{enumerate}[label=(\roman*),itemsep=0pt]
\item The oriented boundary of $X$ can be decomposed as $-Y_-\cup N\cup Y_+$ .
\item $-Y_-\cap Y_+=\emptyset$.
\item $\partial N=-\partial Y_-\amalg \partial Y_+$.
\item $N \cong [-1,1] \times \partial Y_-$ (and thus $\partial Y_- \cong \partial Y_+$).
\item $N\cap Y_\pm=\partial Y_\pm$.
\end{enumerate}
The manifold $X$ is then called a relative cobordism from $Y_-$ to $Y_+$. In addition, the manifold $N$ is sometimes referred to as the horizontal boundary component of $X$, and the manifolds $Y_-$ and $Y_+$ as the vertical boundary components. (We view our cobordisms as running from left to right, not from bottom to top.)
\end{definition}

When $\partial X = Y$ has an open book decomposition, this gives us a relative cobordism structure in the following sense. Consider the open book $(\Lambda, f)$ on $Y$ and restrict the fibration $f$ to the complement $Y \setminus \nu(\Lambda)$ of an open tubular neighborhood of the link $\Lambda$.
An identification of $S^1$ with a square $\partial ([-1,1] \times[-1,1])$, and a decomposition of the square into its vertical and horizontal components $v_\pm = \{\pm 1\} \times [-1,1], h_\pm = [-1,1] \times \{\pm 1\}$ induces the following decomposition on $Y$: $$Y=f^{-1}(v_-\amalg v_+) \cup \left(\nu(\Lambda)\cup f^{-1}(h_-\amalg h_+)\right).$$

Then, if we set $Y_\pm=f^{-1}(v_\pm)$ we call $Y_-$ and $Y_+$ the vertical components of $Y$, and $N=\nu(\Lambda)\cup f^{-1}(g_-\amalg g_+)$ the horizontal component of $Y$. Moreover, notice that if $P$ is a page of the fibration $f$, then $Y_\pm \cong [-1,1] \times P$, and $N\cong \nu(\partial P)\cup \left([-1,1]\times P\amalg [-1,1]\times P\right) \cong [-1,1] \times D(P)$, where $D(P)$ denotes the double of $P$.

With respect to this decomposition of $Y=\partial X$, $X$ is now a relative cobordism from $Y_- \cong [-1,1] \times P$ to $Y_+ \cong [-1,1] \times P$. Our first goal is to turn the given handle decomposition of $X$, which presents $X$ as a cobordism from $\emptyset$ to $Y$, into a handle decomposition presenting $X$ as a relative cobordism from $Y_-$ to $Y_+$. We will call a handle decomposition of the first type (starting with a $0$--handle to which $1$--, $2$-- and $3$--handles are added) a ``standard handle decomposition'', and a handle decomposition of the second type (starting with $[-1,1] \times Y_- \cong [-1,1] \times [-1,1] \times P$ to which $1$--, $2$-- and $3$--handles are added along $\{1\} \times [-1,1] \times P$) a ``relative handle decomposition''. Intermediate between an arbitrary standard handle decomposition and a relative decomposition, we introduce:
%

\begin{definition}\label{compatibility} A standard handle decomposition of a $4$--manifold $X$, involving a single $0$--handle, and some number of $1$--, $2$-- and $3$--handles, is {\em compatible with} a given open book $(\Lambda,f)$ on its boundary $\partial X$ if, for some page $P \subset \partial X$ of the open book, the $0$--handle and some of the $1$--handles forms a neighborhood $[-1,1] \times [-1,1] \times P$ of $P$ in $X$, with $\{-1\} \times [-1,1] \times P$ being a neighborhood of $P$ in $\partial X$, and the remaining $1$--, $2$-- and $3$--handles attached along $\{1\} \times [-1,1] \times P$.
\end{definition}
If a standard handle decomposition of $X$ is compatible with a given open book with page $P$, then  $P$ can be given a handle decomposition involving a single $2$--dimensional $0$--handle and several $2$--dimensional $1$--handles with the following behavior:
\begin{itemize}
\item The $2$--dimensional $0$--handle is completely contained in the $4$--dimensional $0$--handle, and it is disjoint from every attaching sphere of the $4$--dimensional handle decomposition.
\item For each $2$--dimensional $1$--handle $H^2_1$ there is a unique $4$--dimensional $1$--handle $H^4_1$ such that the core of $H^2_1$ intersects the co-core of $H^4_1$ transversely exactly once, i.e. $H^2_1$ goes over $H^4_1$ geometrically once and no other $2$--dimensional $1$--handles go over $H^4_1$.
\end{itemize}

With a given standard handle decomposition of $X$ and a page $P$ visible in $Y = \partial X$, we will most likely need to add cancelling $1$--$2$ pairs in order to achieve this compatibility. In our example, having drawn a picture in which the $2$--dimensional $1$--handles of $P$ are obvious, we see a natural place to put the $1$--$2$ pairs: One of the $2$--dimensional $1$--handles already went over a $4$--dimensional $1$--handle, and now we add two cancelling pairs to accommodate the other two $2$--dimensional $1$--handles. A ``planar diagram'' for $P$ ca then be obtained via an ambient isotopy. In some sense, this is the hardest part of the process to implement in practice. The result of the isotopy is shown in \autoref{fig::adapted}, where the $0$--handle for the page is drawn as a hexagon, and the handles as identifications of some of the sides. 

\begin{figure}[h]
\centering
\def\svgwidth{0.3\textwidth}
\input{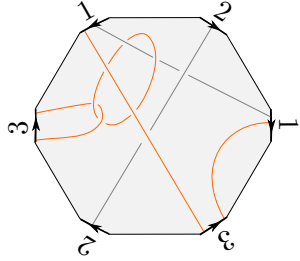}
\caption{Kirby diagram for a handle decomposition of $X$ compatible with the given open book on $\partial X$. The grey $2$--handles were added to cancel the $1$--handles and these, in turn, were added to accommodate the 2--dimensional handles of the page. \label{fig::adapted}}
\end{figure}

%

Returning to the general setting, we now assume that the handle decomposition of $X$ is compatible with the given open book. Let $X_1$ be the union of the $0$--handle and the $1$--handles. Since the handle decomposition of $X$ is compatible with the open book on $\partial X$, then 
$X_1$ can be realized as the union of $[-1,1]^2\times P$ with some number $a_1$ of ``purely $4$--dimensional'' $1$--handles. Thus, $X_1$ is isomorphic to $\natural^{k_1} S^1\times B^3$ and $\partial X_1$ is naturally decomposed into 
$$ \partial_{\operatorname{bdry}} X_1=\partial X\cap X_1=  Y_-\cong [-1,1]\times P,$$ 
and 
\begin{align*}
\partial_{\operatorname{int}} X_1&=\partial X_1\cap\;{\operatorname{int}}(X)\\&=(\{1\}\times Y_-) \setminus (\amalg^{a_1}(S^0\times D^3)) \cup (\amalg^{a_1}(D^1\times S^2)).
\end{align*}
That is, $\partial_{\operatorname{int}} X_1$ is the result of performing surgery on $[-1,1]\times P$ along the attaching spheres of the purely $4$--dimensional handles. To understand the effect of these handles, align the attaching regions of the $1$--handles so that the equator of each $S^0\times \partial D^3$ lies in $\{0\}\times P$. Then, removing a vertical half of $S^0\times D^3$ from each $[-1,0]\times P$ and $[0,1]\times P$, results in a space that is diffeomorphic to $I\times P$ and that has marked disks in its boundary. Additionally, the products $D^1\times S^2$ are attached along the northern hemisphere to $[0,1]\times P\setminus a_1(S^0\times D^3)$ and along the southern hemisphere to $[-1,0]\times P\setminus a_1(S^0\times D^3)$. Thus, $\partial_\text{int}X_1$ has a sutured Heegaard splitting $H_{12}\underset{F}{\cup}H_{13}$ where $F$ is the surface (with boundary) that results from performing surgery on $P$ along the $a_1$ embedded $0$--spheres, and each $H_{ij}$ is the result of attaching $3$--dimensional $1$--handles to $I\times P$.
%

Now let $L$ be the framed attaching link for the $2$--handles; $L$ is in fact embedded in $\partial_\text{int} X_1$, and can be projected onto the surface $\{0\}\times F$, as already illustrated in \autoref{fig::adapted}. If necessary (it is not in our example), use Reidemeister I moves to ensure that the handle framing agrees with the surface framing. Since at each crossing in the projection there is a way of knowing which strand is to be thought of as the one passing under, the two points in $L$ that project into each one of the crossing points can be labeled as `over point' and `under point'. Then, if a diagram of the projection has $c$ crossing points, the link $L$ has a decomposition into $c$ vertices and $c$ edges by placing a vertex at each `under point' of $L$. Call an edge an `over edge' if it contains at least one `over point'. If necessary (again it is not in our example), use Reidemeister II moves to ensure that each component has at least one over edge.
%

The next step involves the stabilizations of the sutured Heegaard splitting. These stabilizations can be performed at each crossing in the decomposition of $L$ induced by the projection, but it is worth noting that we can often be more efficient if the diagram for $L$ is not alternating: if one strand goes over several other understrands in succession, we can resolve this with one stabilization . Specifically, if $e_{ij}$ is the $j$-th over edge in the component $L_i$ of $L$, then form the handlebody $H'_{31}$ by removing from $H_{31}$ the $3$--dimensional tubular neighborhoods of arcs in $H_{31}$ parallel rel. boundary to interior segments of the edges $e_{ij}$. In addition, form the space $H'_{12}$ by adding these tubular neighborhoods to $H_{12}$, and set $F'=H'_{12}\cap H'_{31}$. Next, consider the disk $D_{ij}$ obtained as the isotopy rel. boundary that gives the arc parallel to $e_{ij}$. Finally, resolve the crossings in $L$ by sliding each over strand $e_{ij}$ over $D_{ij}$ and into $F'$. This gives a higher genus sutured Heegaard splitting of $\partial_\text{int} X_1$. Pushing $L$ into the interior of $H'_{12}$, we are ready to produce the trisection (and simultaneously recap the definition of a relative trisection) in much the same way as explained in~\cite[Lemma 14]{gk}: \\
Let $N=[-\epsilon,\epsilon] \times H'_{12}$ be a small tubular neighborhood of $H'_{12}$ with $[-\epsilon,0] \times H_{12} = N \cap X_1$, and set
\begin{itemize}[itemsep=0pt]
 \item $X_1$ as described,
 \item $X_2$  the union of $[0,\epsilon] \times H'_{12}$ and the $2$--handles, and 
 \item $X_3=X\setminus\text{int}\left(X_1\cup X_2\right)$
\end{itemize} 
To see that $X=X_1 \cup X_2 \cup X_3$ is a relative trisection, notice that:
\begin{itemize}[itemsep=0pt]
 \item $F' = X_1 \cap X_2 \cap X_3$ is a compact surface with boundary,
 \item both $H'_{31} = X_3 \cap X_1$ and $H'_{12} = X_1 \cap X_2$ are relative compression bodies compressing $F'$ down to the page $P$,
 \item $X_1$ is diffeomorphic to $\natural^{k_1} (S^1 \times B^3)$ for some $k_1$,
 \item $X_1\cap\partial X\cong I\times P$.
\end{itemize}
It remains to verify that both $X_2$ and $X_3$ are diffeomorphic to $\natural^{k_i} (S^1 \times B^3)$ for some $k_2$ and $k_3$, with both $X_2\cap\partial X$ and $X_3\cap\partial X$ diffeomorphic to $I\times P$, and that $H'_{23} = X_2 \cap X_3$ is a relative compression body from $F'$ to $P$. This will be an unbalanced trisection if $k_1$, $k_2$ and $k_3$ are distinct, but standard trisection stabilizations can balance the trisection to get $k=k_1=k_2=k_3$.

To see that $X_2$ is isomorphic to $\natural^{k_2} (S^1 \times B^3)$, notice that the space $[0,\epsilon] \times H'_{12} \cong \natural^{l} (S^1 \times B^3)$ (for some $l$) inherits a handle structure from that of $H'_{12}$. Moreover, the belt sphere of the $4$--dimensional $1$--handle is the union of $[0,\epsilon]\times \beta_j$ and $\{0,\epsilon\}\times D_j$ where $D_j$ is a compression disk for $H'_{12}$ and $\beta_j$ is its boundary. Notice also that the way the stabilizations were performed guarantees that every component $L_i$ of $L$ intersects the belt spheres $\beta_{ij}$ transversely. So, fix one index $J_i$ for each $i$ to be the belt sphere that $L_i$ intersects transversely, and slide every other $\beta_{ij}$ over $\beta_{i J_i}$ so that it no longer intersects $L_i$. This shows that the $1$--handle with core parallel to the over edge $e_{iJ_i}$ and the $2$--handle with attaching circle $L_i$ are a cancelling pair, and so that $X_2\cong \natural^{k_2} (S^1 \times B^3)$ (for $k_2$ equal to $l$ minus the number of components of $L$). This process will also help us produce enough $\gamma$ curves to get a trisection diagram.

Turning the relative handle decomposition ``upside-down'' shows us that $X_3=X\setminus{\operatorname{int}}\left(X_1\cup X_2\right)$ is diffeomorphic to $\natural^{k_3} (S^1 \times B^3)$, $X_3\cap\partial X$ is diffeomorphic to $[-1,1]\times P$, and the interior boundaries $\partial_{\operatorname{int}}X_1$ and $\partial_{\operatorname{int}}X_3$ are also diffeomorphic. Since the attaching link $L$ can be isotoped to lie in the interior of $H'_{12}$, the latter implies that $H'_{23}=\partial_{\operatorname{int}} X_2\setminus H'_{12}$, the result of surgery on $H'_{12}$ along the attaching circles of the $2$--handles, is diffeomorphic to $H'_{12}$ and thus to the desired compression body.

We return to the example in \autoref{fig::adapted}. Notice that the attaching circles of the $2$--handles have already been projected into $P$, with no Reidemeister I or II moves needed. Also, there are no ``purely $4$--dimensional'' $1$--handles, and no $3$--handles, i.e. the relative handle decomposition has only $2$--handles. We need to stabilize the page five times to resolve the crossings in the $2$--handle link $L$. (There areeight crossings, but some of them can be resolved in groups with a single stabilization since the diagram is not alternating.) Thus we get a trisection immediately in which:
\begin{itemize}
 \item $X_1$ is $[-1,1] \times [-1,1] \times P$, where $P$ is the page, a genus $1$ surface with two boundary components. Thus $X_1 \cong \natural^3 S^1 \times B^3$
 \item $\partial_{\operatorname{int}} X_1 \cong [-1,1] \times P$ is split along a genus $6$ surface $\Sigma$ with two boundary components, obtained by stabilizing the splitting along $0 \times P$ five times. This splits $\partial_{\operatorname{int}} X_1$ into two compression bodies $H_{31}$ and $H_{12}$, each of which compresses $\Sigma$ down to $P$ by compressing along $5$ simple closed curves. Thus, ignoring the sutures, each of $H_{12}$ and $H_{31}$ is a genus $8$ handlebody.
 \item $X_2$ is a thickening of $H_{12}$ together with the four $2$--handles, each of which cancels one of the eight $1$--handles in $I \times H_{12} \cong \natural^8 S^1 \times  B^3$, so that $X_2 \cong \natural^4 S^1 \times B^3$.
 \item Thus $H_{23}$ is also a compression body from $\Sigma$ to $P$, and since there are no $3$--handles, $X_3$ is again $[-1,1] \times [-1,1] \times P \cong \natural^3 S^1 \times B^3$.
\end{itemize}
This is an unbalanced trisection, and we can draw its diagram by understanding which curves on the genus $6$ surface $\Sigma$ bound disks in the three compression bodies. In the two compression bodies forming the boundary of $X_1$, namely $H_{31}$ and $H_{12}$, these are the five standard $\alpha$ (red, bounding in $H_{31}$) and $\beta$ (blue, bounding in $H_{12}$) curves corresponding to the five stabilizations, indicated in \autoref{fig::trisection}. The four components of the attaching link $L$ for the $2$--handles give four of the five curves which bound disks in $H_{23}$ (green $\gamma$ curves). To find the fifth $\gamma$ curve, we note that the four curves coming from $L$ cancel four of the $\beta$ curves, so we can carry out this cancellation and the missing $\gamma$ curve can be obtained as a parallel copy to the curve obtained as $\beta_C+\beta_D$. The result is shown in \autoref{fig::trisection}. Finally we can stabilize to get a balanced trisection if desired, also shown in \autoref{fig::trisection}.
\end{proof}
\begin{figure}[h]
\centering
\fontsize{8pt}{8pt}
\def\svgwidth{0.5\textwidth}
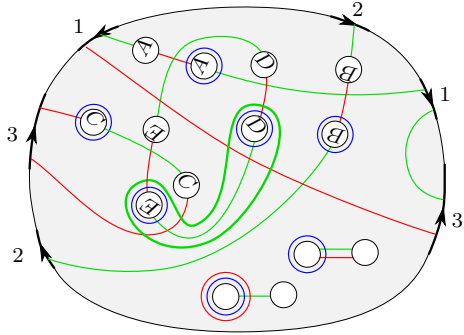
\caption{ A $(8,5;1,2)$ relative trisection diagram for the exterior of the standard torus in $S^4$. The thick green curve is the $\gamma$ curve that is not an attaching circle of the $2$--handles.}
\label{fig::trisection}
\end{figure}
\end{section}

\begin{section}{Trisections of $4$--manifolds with Multiple Boundary Components\label{sec:basicdefs}}

We extend the definition of relative trisections to manifolds with $m>1$ boundary components by generalizing the construction in \cite{cgp}. For this we require two integers $g,k \geq 0$ and $m$ pairs of integers $(p_1, b_1), \ldots, (p_m, b_m)$ all of which satisfy $\sum(2p_i + b_i -1)\leq k \leq g - \sum(p_i + b_i -1)$. Let us fix the following notation: $p = \sum p_i$, $b=\sum b_i$, $l_i = 2p_i + b_i -1$, and $l = \sum l_i$. Define $U_i= P_i \times D$, where $D = \{re^{i\theta}| 0 \leq r \leq 1; -\pi/3 \leq \theta \leq \pi/3\}$ is a third of the unit disk, whose boundary we decompose into three pieces:
\begin{align*}
\partial ^0 D &=\{e^{i\theta}| - \pi/3 \leq \theta \leq \pi/\theta\}\cr
\partial^\pm D &= \{re^{\pm i\pi/3}|0\leq r\leq 1\}.
\end{align*}
Each $\partial U_i$ has the corresponding decomposition
\begin{align*}
\partial^0U_i	&= (P_i \times \partial D) \cup (\partial P_i \times D)\cr
\partial^\pm U_i	&= P_i \times \partial^\pm D.
\end{align*}
Note that $U_i\cong\natural^{l_i}S^1\times B^3$. We now connect each of the $U_i$'s to obtain
\begin{align*}
U &= U_1 \natural \cdots \natural U_m = \natural^l S^1\times B^3.
\end{align*}
These boundary connected sums are again taken to preserve the splittings, giving us
$$\partial^0U = \coprod_{i=1}^m\partial^0U_i\;\text{ and }\;
\partial^\pm U = \partial^\pm U_1 \natural \cdots \natural \partial^\pm U_m.$$
We now proceed as before, attaching $U$ to $V_n:=\natural^nS^1\times B^3$, whose boundary has the unique genus $n+s$ Heegaard splitting $\partial_sV_n = \partial^+_sV_n \cup \partial^-_sV_n.$ Choosing $n=l-k$ and $s=g-n-p$ gives us
\begin{equation*}
Z_k	:=U\natural V_n \cong \natural^k S^1\times B^3
\end{equation*}
whose boundary $\partial Z_k \cong \#^k S^1\times S^2$ inherits the decomposition
\begin{align*}
\partial Z_k &:= Y_{g,k} = Y_{g,k}^+ \cup Y_m^0 \cup Y_{g,k}^-,
\end{align*}
where the subscript of $Y^0_m=\partial^0U$ indicates that the trisection is of a manifold with $m>1$ boundary components.  Although the notation is the same, we have

$$Y_{g,k}^\pm	= (\partial^\pm U_1 \natural \cdots \natural \partial^\pm U_m) \natural \partial^\pm_sV_n.$$

\begin{definition}
A \emph{$(g,k;p_1,b_1, \ldots, p_m, b_m)$ relative trisection} of a compact, connected, smooth, oriented $4$--manifold $X$ with $m>0$ boundary components is a decomposition of $X$ into three codimension 0 submanifolds $X=X_1 \cup X_2 \cup X_3$ such that for each $i=1,2,3$,
\begin{enumerate}[itemsep=0pt]
\item there exists a diffeomorphism $\varphi_i: X_i \rightarrow Z_k$
\item $\varphi_i(X_i \cap X_{\pm 1}) = Y^\pm_{g,k}$
\item $\varphi_i(X_i \cap \partial X) = Y^0_m,$
\end{enumerate}
where indices are taken mod $3$. 
\end{definition}

%

\begin{lemma}\label{lem:obd}

A $(g,k;p_1,b_1;\ldots;p_m,b_m)$ relative trisection of a $4$--manifold $X$ with $m>0$ boundary components induces an open book decomposition on each component $\partial_i X$ of $\partial X$ whose page is $P_i$, a genus $p_i$ surface with $b_i$ boundary components. 

\end{lemma}

\begin{proof}

Suppose $X=X_1 \cup X_2 \cup X_3$ is a $(g,k;p_1, b_1;\ldots; p_m, b_m)$ relative trisection. Recall that $X_i \cap \partial X = Y^0_m \cong (P \times I)\cup(\partial P \times D)$, where $P = \underset{i=1}{\overset{m}{\sqcup}} P_i$ and each $P_i$ is a genus $p_i$ surface with $b_i$ boundary components. We also denote $D_j = \{re^{i\theta}|\; |r|\leq 1, \; \frac{2\pi(j-1)}{3} \leq \theta \leq \frac{2\pi j}{3}\}$, for $j=1,2,3$, to be a third of the unit disk associated to each $X_j$. Thus,
\begin{align*}
\partial X = & (X_1\cap \partial X)\cup(X_2 \cap \partial X)	\cup(X_3\cap \partial X)\cr
= & \left((P \times I_1)\cup(P\times I_2)\cup(P\times I_3)\right)\cr
& \bigcup \left( (\partial P \times D_1)\cup (\partial P \times D_2)\cup (\partial P \times D_3)\right)\cr
= & (P \underset{\mu}{\times} S^1) \cup (\partial P \times D^2)
\end{align*}
where $P \underset{\mu}{\times} S^1$ is the $m$--component mapping cylinder of the monodromy map $\mu$, and $\partial P \times D^2$ is regular neighborhood of the binding  $\partial P$. Restricting $\mu$ to each $P_i$ gives an open book decomposition, $(P_i \underset{\mu}{\times}I) \cup (\partial P_i \times D^2)$, on each component of $\partial X$ as desired.
\end{proof}

\begin{theorem}\label{thm:existence}

Given an open book decomposition on each component of $\partial X$, there exists a relative trisection of $X$ inducing the given open book(s).

\end{theorem}

The proof in this case is essentially identical to the connected boundary component case proved earlier; rather than replicate that proof, we see how it works in this case in \autoref{sec:products} by turning an interesting handle decomposition of $M \times I$ into a trisection diagram.

\end{section}

\begin{section}{Relative Trisection Diagrams}

Trisection diagrams allow us to describe any $4$--manifold in terms of three collections of curves $\alpha, \beta, \gamma$ on a surface $\Sigma$. The most notable characteristic of relative trisection diagrams for manifolds with multiple boundary components is that each collection $\alpha, \beta, \gamma$ contains (at least one) separating curve. This is a requirement since performing surgery on $\Sigma$ along any one of $\alpha, \beta, \gamma$ results in $m>0$ surfaces with boundary, each one corresponding to a page of an open book induced on a component of $\partial X$.

\begin{definition}

A \emph{$(g,k;p_1, b_1; \ldots; p_m, b_m)$ relative trisection diagram} is a $4$--tuple $(\Sigma, \alpha, \beta, \gamma)$ such that

\begin{enumerate}[itemsep=0pt]

\item $\Sigma$ is a genus $g$ surface with $b$ boundary components,

\item each of $\alpha=\{\alpha_1, \ldots, \alpha_{g-p}, \widetilde{\alpha}_1, \ldots, \widetilde{\alpha}_{m-1}\}$ is a collections of $g-p+m-1$ disjoint, simple, closed curves such that $\alpha_1, \ldots, \alpha_{g-p} \subset \alpha$ are essential, and each $\widetilde{\alpha}_1, \ldots, \widetilde{\alpha}_{m-1} \subset \alpha$ is a separting curve;  $\beta$ and $\gamma$ are similarly defined, and

\item each triple $(\Sigma, \alpha, \beta), (\Sigma, \beta, \gamma), (\Sigma, \alpha, \gamma)$ is handleslide diffeomorphic to the standard diagram $(\Sigma, \epsilon, \delta)$ in \autoref{fig:standard}.

\end{enumerate}

\end{definition}

\begin{figure}[h!]{
\centering
\includegraphics[scale=.35]{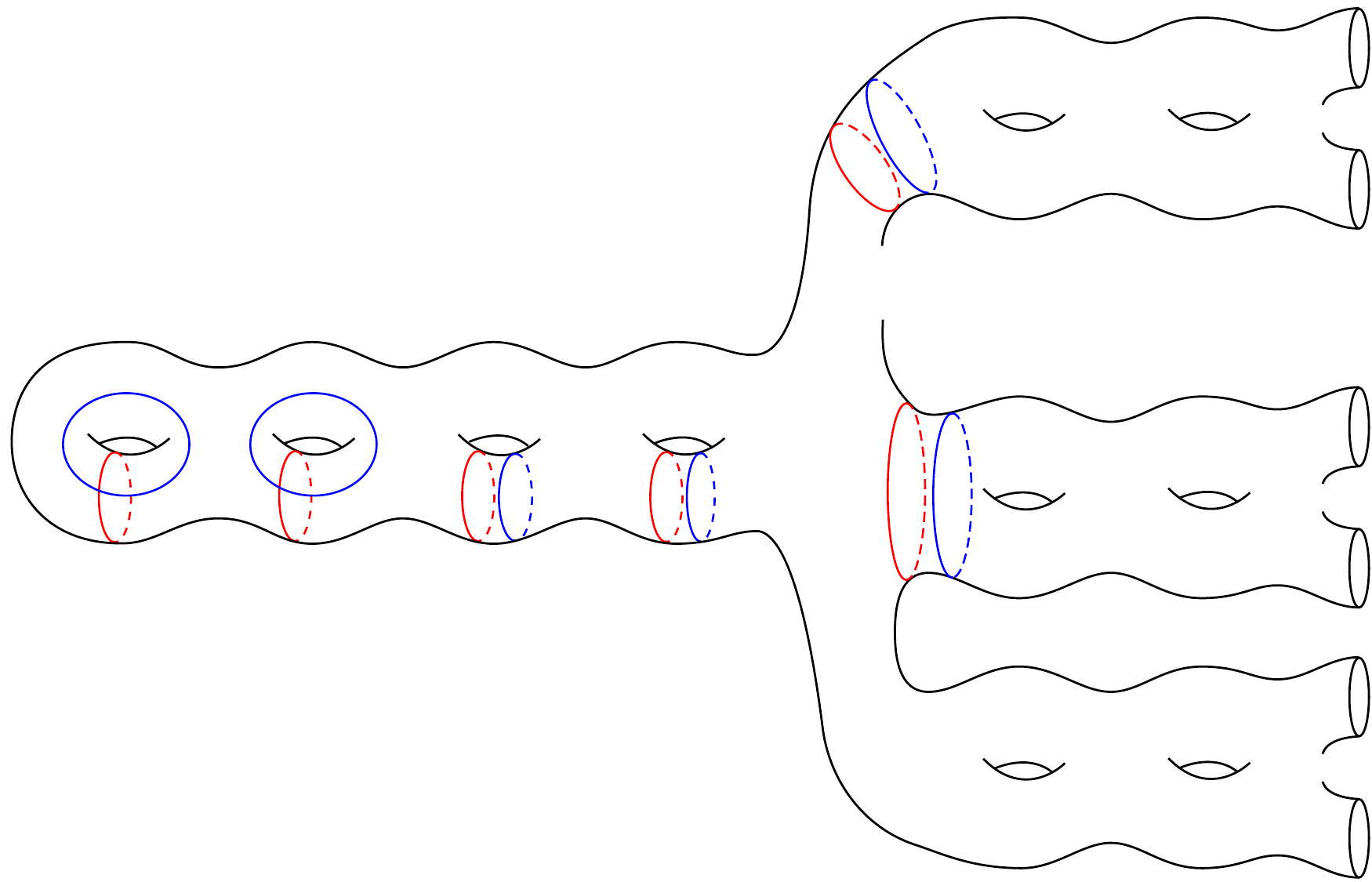}
\caption{Standard Position of $(\Sigma, \epsilon, \delta)$. There are $m-1$ separating curves, $\widetilde{\epsilon}_1, \ldots, \widetilde{\epsilon}_{m-1}$ and $\widetilde{\delta_1}, \ldots, \widetilde{\delta}_{m-1}$\label{fig:standard}. Each $\widetilde{\epsilon}_i$ and $\widetilde{\delta}_i$ separates $\Sigma$ into two pieces, one with curves and one without. The latter corresponds to the page of the induced open book on a boundary component of the corresponding $4$--manifold.}}
\end{figure}

\begin{theorem} There is a one-to-one correspondence between relatively trisected $4$--manifolds up to diffeomorphism and relative trisection diagrams up to handleslides and diffeomorphisms of the trisection surface. \end{theorem}

\begin{proof}

To obtain a relative trisection diagram from a relatively trisected $4$--manifold, we recall the above construction of each $X_i$ in the trisection. We define $\Sigma$ to be a genus $g$ surface with $b$ boundary components, diffeomorphic to the trisection surface $X_1\cap X_2 \cap X_3$. Define $\alpha\subset \Sigma$ to be the simple closed curves corresponding to the boundaries of compressing disks in the compression body $X_1\cap X_2$; similarly, let $\beta$ correspond to the boundaries of compression disks in $X_2 \cap X_3$, and $\gamma$ correspond to those in $X_1\cap X_3$.

Conversely, suppose $(\Sigma, \alpha, \beta, \gamma)$ is a $(g,k;p_1, b_1; \ldots; p_m,b_m)$ relative trisection diagram. We follow the construction in \cite{gk} for closed $4$--manifolds, which was adapted to relative trisections with connected boundary in \cite{cgp}. Let $\Sigma_\alpha$ denote the collection of $m$ surfaces obtained by performing surgery on $\Sigma$ along $\alpha$; $\Sigma_\beta$ and $\Sigma_\gamma$ are similarly defined. Note that $\Sigma_\alpha \cong \Sigma_\beta \cong \Sigma_\gamma \cong P := \underset{i=1}{\overset{m}{\sqcup}} P_i$, where each $P_i$ is a fixed genus $p_i$ surface with $b_i$ boundary components. Let us fix identifications of each of our surgered surfaces with $P$. Let $C_\alpha, C_\beta, C_\gamma$ be the $3$--dimensional cobordisms, known as compression bodies, from $\Sigma$ to $\Sigma_\alpha, \Sigma_\beta, \Sigma_\gamma$ respectively. Since $\Sigma$ is a surface with boundary, the ``ends'' of our cobordism $\Sigma\sqcup \Sigma_\alpha$ only form a portion of $\partial C_\alpha$. The remainder of $\partial C_\alpha$ is the ``sides'' of the cobordism, $\partial_s C_\alpha \cong \partial P\times [0,1].$

Let $a,b,c \in\partial B^2$ with disjoint, closed regular neighborhoods $N_a, N_b, N_c \subset \partial B^2$, each of which can be identified with the unit interval. Let $I_{ab} \subset \partial B^2$ be the curve whose boundary is ${a,b}$; $I_{bc}$ and $I_{ac}$ are similarly defined. Attach $C_\alpha \times I$ to $\Sigma \times B^2$ so that $\Sigma \times I \subset C_\alpha \times I$ is identified with $\Sigma \times N_a$. We similarly attach $C_\beta \times I$ and $C_\gamma \times I$ along $\Sigma \times N_b$ and $\Sigma \times N_c$ respectively. \autoref{fig:tripod} gives a schematic of our current intermediate $4$--manifold, denoted by $X^\circ$.

\begin{figure}[h!]
\centering
\fontsize{7pt}{7pt}
\def\svgwidth{0.4\textwidth}
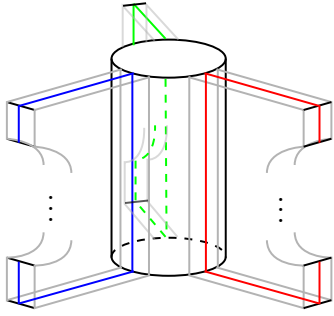

\caption{The intermediate stage of constructing a unique $4$--manifold with $m$ boundary components from a relative trisection diagram. $\Sigma \times B^2$ is depicted by the central cylinder from which the three compression bodies protrude.
\label{fig:tripod}}

\end{figure}

Having fixed identifications of $\Sigma_\alpha$ and $\Sigma_\beta$ with $P$, there is a correspondence between the components of $\Sigma_\alpha$ and $\Sigma_\beta$. For each $i=1, \ldots, m$, we attach $P_i \times I \times I$ to $X^\circ$ so that $P_i \times I \times \{0\}$ is attached to the corresponding component of $\Sigma_\alpha \times I$ in the $C_\alpha$ protrusion of $X^\circ$ and $P_i\times I \times\{1\}$ is attached to the corresponding component of $\Sigma_\alpha\times I$ in the $C_\beta$ protrusion. Note that a correspondence exists between any pair of $\Sigma_\alpha, \Sigma_\beta$, and $\Sigma_\gamma$. As such, we attach $C_\beta$ to $C_\gamma$ and $C_\gamma$ to $C_\alpha$ with copies of $P\times I \times I$. We refer the reader to \cite[Corollary 14]{cgp} for the proof that there is a unique way to attach each $P_i \times I \times I$. Attaching the three thickened copies of $P_i \times I$ induces a surface bundle over $S^1$. Furthermore, notice that for each $i$ there is a copy of $\partial P_i \times S^1$ in the boundary of our smoothable $4$--manifold: 
\begin{equation}\label{eq:1}
\partial P_i \times \left( \partial_s C_\alpha \cup I_{ab} \cup \partial_s C_\beta\right).
\end{equation} 
Fixing an identification of \ref{eq:1} with $\partial P_i \times S^1$, we attach disjoint solid tori $\underset{j=1}{\overset{b_i}{\sqcup}} (S^1 \times \partial B^2)$, one for each boundary component of $P_i$, so that $\{x\} \times S^1$ bounds a disk for each $x \in \partial P_i$. These disks correspond to a third of the unit disk, as discussed in \autoref{sec:basicdefs}, and together comprise a neighborhood of $\partial P$.

In the case of $m=1$, the resulting $4$--manifold has four boundary components: one boundary component is a $3$--manifold equipped with an open book decomposition, and three boundary components are diffeomorphic to $\#^k S^1 \times S^2$ for some $k>0$. We then implement a result of Laudenbach and Poenaru which states that there is a unique way to fill $\#^k S^1\times S^2$ with $\natural ^kS^1 \times B^3$ \cite{LP} to obtain a $4$--manifold whose connected boundary is equipped with an open book decomposition by construction. For $m>1$, we have $m$ boundary components equipped with open books, and $3$ boundary components diffeomorphic to $\#^{k} S^1 \times S^2$. Again, these three components can be uniquely filled to obtain a $4$--manifold $X$ with $m$ boundary components, each accompanied with an induced open book.

\end{proof}

\end{section}




\begin{section}{Gluing Relative Trisection Diagrams}

Before discussing the gluing theorem, we first recall the algorithm described in~\cite[Theorem 5]{cgp} for determining the monodromy of an open book determined by a relative trisection diagram; this works without modification in the case of multiple boundary components and is essential for drawing the diagram of a trisection obtained by gluing two relatively trisected $4$--manifolds along diffeomorphic boundary components. 

Recall that $\Sigma_\alpha$ is the surface obtained from $\Sigma$ by surgery along $\alpha$. A collection $a$ of properly embedded arcs in $\Sigma$ disjoint from $\alpha$ is called a \emph{cut system of arcs for $(\Sigma,\alpha)$} if $a$ descends to a collection $a_\alpha$ of properly embedded arcs in $\Sigma_\alpha$ which cut each component of $\Sigma_\alpha$ into a disk. 

%
%
%
%
%

The algorithm for obtaining the monodromy for the induced open book decomposition gives a well defined procedure, given a cut system of arcs $a$ for $\alpha$, to produce cut systems of arcs $b$, $c$ and $\hat{a}$ for $\beta$, $\gamma$ and $\alpha$, respectively, such that the monodromy map $\mu:\Sigma_\alpha \rightarrow \Sigma_\alpha$ is the unique map such that $\mu(a_\alpha)= \hat{a}_\alpha$; $\mu$ actually factors as a map $\Sigma_\alpha \to \Sigma_\beta \to \Sigma_\gamma \to \Sigma_\alpha$ taking $a_\alpha$ to $b_\beta$ to $c_\gamma$ to $\hat{a}_\alpha$.

(An astute reader might notice that the algorithm in ~\cite{cgp} produces a cut system of arcs $b$ for a system of curves which is handle slide equivalent to $\beta$, not for $\beta$ itself, and similarly for $c$ in relation to $\gamma$ and $\hat{a}$ in relation to $\alpha$. However, each arc system can be turned back into an arc system for the original system of curves since, any time one slides a closed curve $x$ over another curve $y$, if an arc $z$ from the corresponding arc system gets in the way, then $z$ can first be slid over $y$ to get it out of the way.)

Suppose $(\Sigma, \alpha,\beta, \gamma)$ and $(\Sigma',\alpha', \beta',\gamma')$ are relative trisection diagrams which correspond to $4$--manifolds $X$ and $X'$ with diffeomorphic boundaries. It is natural to ask how these relative trisection diagrams relate to a relative trisection diagram of the closed $4$--manifold $X \underset{\partial}{\cup} X'$. Of course the induced structures on the bounding $3$--manifold must be compatible. That is, if $(P, \mu)$ and $(P', \mu')$ are the open books induced by $(\Sigma, \alpha,\beta, \gamma)$ and $(\Sigma',\alpha', \beta',\gamma')$ respectively, there must exist an orientation reversing diffeomorphism $f:P \to P'$ such that  $f\circ\mu\circ f^{-1}=\mu'$. The function $f$ gives an identification between $\partial \Sigma$ and $\partial \Sigma'$, which we use to define the closed genus $G=g+g'+b-1$ surface $S = \Sigma \underset{\partial}{\cup} \Sigma'$. Once we glue, we omit the separating curves $\widetilde{\alpha}, \widetilde{\beta},  \widetilde{\gamma}$. The resulting collections $\alpha\cup\alpha', \beta\cup\beta', \gamma\cup\gamma'$ each contain $G-l$ curves. To account for the missing curves, we use the cut systems of arcs. Since we have the freedom to choose our initial cut systems of arcs $a\subset\Sigma$ and $a'\subset\Sigma'$, we choose $a'$ so that $\partial a'_i=f(\partial a_i)$. We denote the newly formed curves $\alpha^*_i=a_i \underset{\partial}{\cup} a'_i$. From $a$, we obtain $b$ and $c$ by using the algorithm mentioned above, in~\cite[Theorem 5]{cgp}, and similarly obtain $b'$ and $c'$, and define $\beta^*_i=b_i \cup b'_i$ and $\gamma^*_i=c_i\underset{\partial}{\cup}c'_i$. We now have three collections of $G$ curves on $S$, \begin{equation}\label{eq:curves-glue}\overline{\alpha}=\alpha \cup \alpha'\cup \alpha^*,\: \overline{\beta}=\beta \cup\beta'\cup\beta^*,\:\overline{\gamma}=\gamma\cup\gamma'\cup\gamma^*.\end{equation} In this case, when we glue along \emph{every} induced open book decomposition, we obtain a $(G,K)$--trisection diagram, where $K=k+k'-l$. 

The gluing theorem applies in greater generality, allowing us to glue together trisection (diagrams) along \emph{any number} of boundary components (possibly less than $m$), resulting in another relative trisection (diagram). Although the indexing becomes much more involved, the general idea remains the same. We glue the surfaces $\Sigma$ and $\Sigma'$ together along boundary components which correspond to the bindings of compatible open book decomposition, using the associated cut system of arcs to account for the missing number of curves on the resulting surface.

\begin{theorem} The tuple $(S, \overline{\alpha}, \overline{\beta}, \overline{\gamma})$, with $\overline{\alpha}, \overline{\beta}, \overline{\gamma}$ as in \ref{eq:curves-glue}, is a trisection diagram for $X \cup X'$. \end{theorem}

For details in the diagrammatic case of $m=1$, the reader is referred to \cite{co}. When gluing relatively trisected manifolds with $m>1$ see \cite{castro}. Here we sketch the general proof.

\begin{proof}
First we show that $(S, \overline{\alpha}, \overline{\beta}, \overline{\gamma})$ is a trisection diagram. This follows from the fact that the Heegaard diagram with arcs $(\Sigma,\alpha \cup a, \beta \cup b)$ is handle slide equivalent to a standard diagram with arcs (where closed curves slide over closed curves and arcs slide over closed curves), and the same is true for the other two diagrams with arcs on $\Sigma$, and the three diagrams with arcs on $\Sigma'$. The standardizing slides then fit together to give standardizing slides for the three pairs $(S, \overline{\alpha}, \overline{\beta})$, $(S, \overline{\beta}, \overline{\gamma})$ and $(S, \overline{\gamma}, \overline{\alpha})$. Having established that we have a legitimate diagram, the fact that it contains as subdiagrams relative diagrams for $X$ and $X'$ shows that the closed $4$--manifold is in fact $X \cup X'$.
\end{proof}

\end{section}
\section{Trisecting product cobordisms}

\label{sec:products}

Note that the gluing theorem allows us to define a more refined bordism category, $\textbf{TRI}$, whose objects are $3$--manifolds equipped with open book decompositions, and morphisms are relatively trisected cobordisms. A curious feature, however, is that a product $[-1,1] \times M^3$ does not have an obvious ``product'' trisection. Here we show how to draw (unbalanced) trisection diagrams for such products. Given a closed, oriented $3$--manifold $M$ with an open book decomposition with page $P$ and monodromy $\mu: P \to P$, fix a $2$--dimensional handle decomposition of $P$ with one $0$--handle and $l=1-\chi(P)$ $1$--handles. Let $\{a_i\}$ be the co-cores of the $1$--handles (properly embedded arcs in $P$) and let $b_i = \mu(a_i)$ be their images under the monodromy. The diagram will be described in terms of this data.

Let $P'$ and $P''$ be copies of $P$ in the interior of $P$ obtained as deformation retracts along collar neighborhoods of $\partial P$, with the property that the $0$--handle of $P''$ is contained in the interior of the $0$--handle of $P'$ but the $1$--handles of $P''$ are disjoint from the $1$--handles of $P'$. Let $D$ be the $0$--handle of $P''$. Let $b'_i$ be the arcs $b_i$ as properly embedded in $P'$, isotoped rel. boundary so as to avoid $D$. Let $c_i$ be the cores of the $1$--handles of $P''$, with endpoints on $\partial D$. See \autoref{F:PageSetup}; in this example the $3$--manifold $M$ is $S^3$.
\begin{figure}[h!]
\centering
%
%
\includegraphics[scale=0.7]{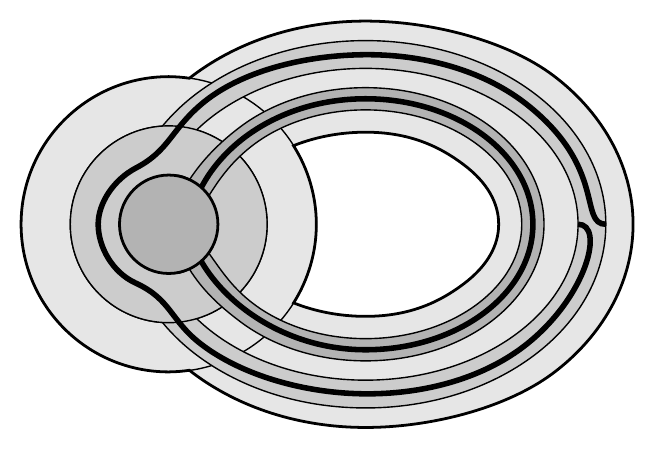}
\caption{Setting up the page. \label{F:PageSetup}}
\end{figure}

We will first draw a Kirby diagram for $[-1,1] \times M$, seen as a cobordism from $\{-1,1\} \times [0,1] \times P$ to $\{-1,1\} \times [0,1] \times P$. This diagram is {\em drawn on} $\{-1,1\} \times P = -P \amalg P$. The diagram has two $B^3$'s for the feet of a single $4$--dimensional $1$--handle, a framed link of $2l$ components, and a $3$--handle which is not drawn. The two $B^3$'s intersect $-P$ and $P$, respectively, as the disk $D$, the $0$--handle of $P''$. Of the $2l$ components of the $2$--handle attaching link, $l$ of them appear as two copies of each of the cores $c_1, \ldots, c_l$ of the $1$--handles of $P''$, each component running over the $4$--dimensional $1$--handle twice. This much is indicated in \autoref{F:KirbyDiagramPt1}.
\begin{figure}[h!]
%
%
\includegraphics[width=0.5\textwidth]{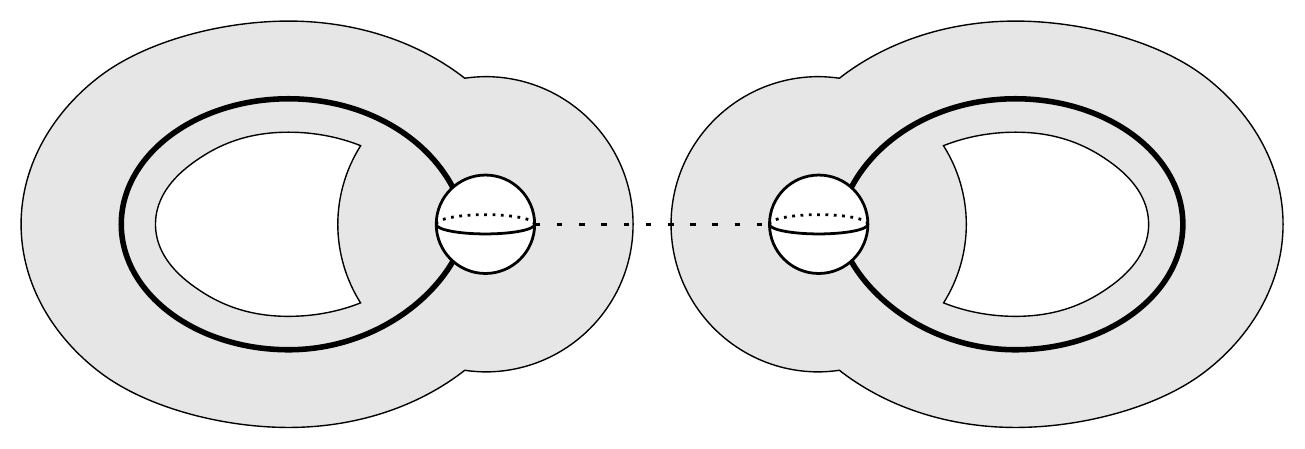}
\caption{The beginning of the Kirby diagram. \label{F:KirbyDiagramPt1}}
\end{figure}
The remaining $l$ components are obtained by completing each monodromy arc $b'_i$ so as to wrap around the $i$-th 1--handle of $P''$. Specifically, we take the union of $b'_i$ in $P' \subset P \subset -P \amalg P$, and an arc going around and behind the corresponding core curve $c_i$ as drawn in \autoref{F:KirbyDiagramPt2}.
\begin{figure}[h!]
\includegraphics[width=0.5\textwidth]{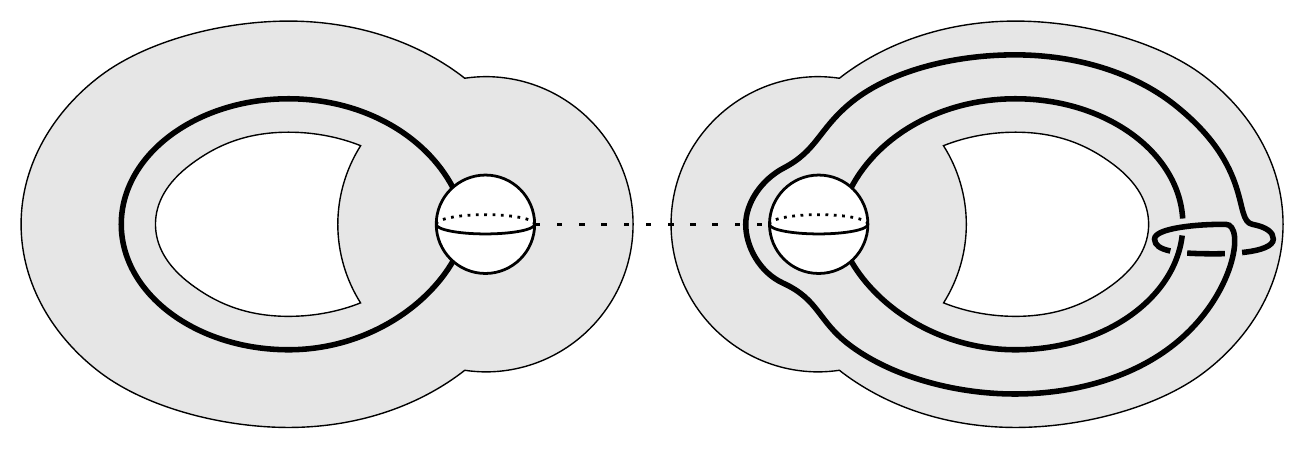}
\caption{The complete Kirby diagram. \label{F:KirbyDiagramPt2}}
\end{figure}
(Note that the ``tilt'' of the plane on which the arc going around and behind lies is important in order to arrange that the framing on this link component is exactly the blackboard framing as drawn.)

Before continuing to the trisection diagram we should discuss the meaning of this diagram briefly. This is a handle attaching diagram for a $4$--manifold built by attaching one $1$--handle, $2l$ $2$--handles and one $3$--handle to $([-2,-1] \amalg [1,2]) \times [0,1] \times P$ along $\{-1,1\} \times [0,1] \times P$, and the diagram is projected onto $\{-1,1\} \times P = -P \amalg P$. This will become a trisection diagram by turning $-P \amalg P$ into $-P \# P$, with the connected sum occuring along a tube running over the $4$--dimensional $1$--handle, and then adding extra genus (extra torus summands) to this surface so as to accommodate the crossings in the $2$--handle attaching link. The connected sum tube gets an $\alpha$ and a $\beta$ curve, parallel, since this records the $4$--dimensional $1$--handle. This much is shown in \autoref{F:TriDiagramPt1}.
\begin{figure}[h!]
\centering
\fontsize{7pt}{7pt}
\def\svgwidth{0.6\textwidth}
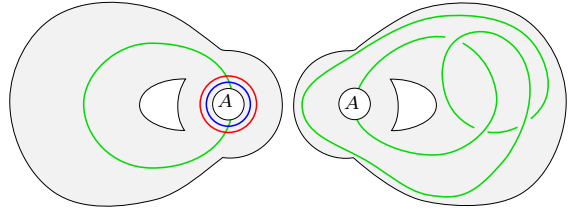
\caption{First step towards the trisection diagram. \label{F:TriDiagramPt1}}
\end{figure}
Each extra torus summand, coming from a stabilization of a Heegaard splitting, gets an $\alpha$ meridian and a $\beta$ longitude curve. Then the $2l$ components of the $2$--handle attaching link become $\gamma$ curves, as shown in \autoref{F:TriDiagramPt2}.
\begin{figure}[h!]
\centering
\fontsize{7pt}{7pt}
\def\svgwidth{0.6\textwidth}
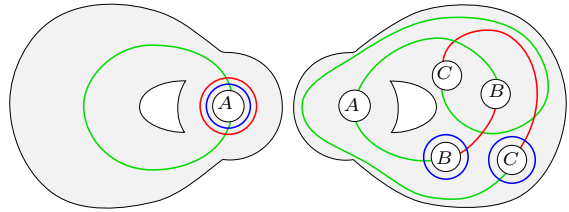
\caption{Second step towards the trisection diagram; we have some but not all of the $\gamma$ curves. \label{F:TriDiagramPt2}}
\end{figure}
The last part is to find the remaining $\gamma$ curves. In the example drawn, we only need one extra $\gamma$ curve. In general, each $\gamma$ curve we already have is dual to some $\beta$ curves, and the remaining $\gamma$ curves should be parallel to the remaining $\beta$ curves. However, this may not be possible with the $\beta$ curves as drawn, so we we need to slide the remaining $\beta$ curves over other $\beta$ curves until they become disjoint from given $\gamma$ curves, and then turn the resulting $\beta$ curves into $\gamma$ curves. In our example, the $\gamma$ curve obtained as two copies of the core $c_1$ is ``dual'' to both $\beta_A$  and $\beta_B$. The remaining $\gamma$ curve is obtained as a parallel copy of the result of sliding $\beta_A$ over $\beta_B$ twice with opposite signs. The final result is shown in \autoref{F:TriDiagramPt3}.
\begin{figure}[h!]
\centering
\fontsize{7pt}{7pt}
\def\svgwidth{0.6\textwidth}
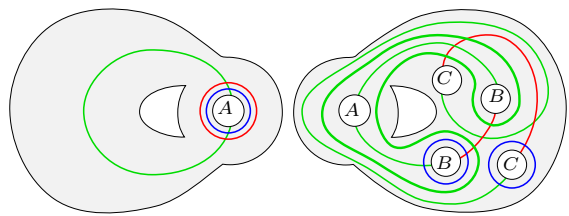
\caption{The complete trisection diagram. \label{F:TriDiagramPt3}}
\end{figure}

%

\bibliographystyle{plain}
\bibliography{RelTriPNAS}

\end{document}